\documentclass[12pt,a4paper]{amsart}

\usepackage[utf8]{inputenc}	

\usepackage[noadjust]{cite}
\usepackage{url} 

\usepackage{amsmath}  
\usepackage{amssymb}  
\usepackage{amsthm}   

\usepackage{graphicx} 
\usepackage{microtype} 

\usepackage{hyperref} 

\numberwithin{equation}{section}
\theoremstyle{plain}
\newtheorem{theorem}  [equation]{Theorem}

\newtheorem{prop}     [equation]{Proposition}

\theoremstyle{definition}
\newtheorem{dfn}      [equation]{Definition}

\theoremstyle{remark}
\newtheorem{rem}      [equation]{Remark}

\DeclareMathOperator	{\Homom}	{Hom}
\DeclareMathOperator	{\End}		{End}
\DeclareMathOperator	{\Mod}		{mod}

\newcommand	{\Ho}[2][]	{\ensuremath{\Homom_{#1}(#2)}}	
\newcommand	{\m}[2]		{\ensuremath{\mu_{#1}(#2)}}
\newcommand	{\my}[1]	{\ensuremath{\mu_k(#1)}}
\newcommand	{\K}		{\ensuremath K}			
\newcommand	{\G}		{\ensuremath{\Gamma}} 		
\newcommand	{\Qu}		{\ensuremath Q} 		
\newcommand	{\tq}[2]	{\ensuremath{\Qu_{#1}(#2)}}		
\newcommand	{\cq}[1]	{\ensuremath{\Qu_{\mathcal C}(#1)}}	
\newcommand	{\cc}		{\ensuremath{\mathcal C}}		

\title[Subquivers of mutation-acyclic quivers]{Subquivers of mutation-acyclic quivers are mutation-acyclic}
\author{Matthias Warkentin}
\date{\today}
\address{Fakultät für Mathematik, Technische Universität 
Chemnitz, D-09107 Chemnitz, Germany}
\email{matthias.warkentin@mathematik.tu-chemnitz.de}
 
\begin{document}

\begin{abstract}
 Quiver mutation plays a crucial role in the definition of cluster algebras by Fomin and Zelevinsky. It induces an equivalence relation on the set of all quivers without loops and two-cycles. A quiver is called mutation-acyclic if it is mutation-equivalent to an acyclic quiver. This note gives a proof that full subquivers of mutation-acyclic quivers are mutation-acyclic.
\end{abstract}

\maketitle

\section{Introduction}
Quiver mutation (see Definition \ref{dfn:quivmut}) was introduced together with cluster algebras by Fomin and Zelevinsky in \cite{cluster1} in connection with total positivity and canonical bases. This led to the development of a quickly growing theory with many interesting connections to other fields, see \cite{cdm-notes} for a nice survey article. Of special interest is a deep relation with representation theory obtained by what is called \textit{categorification} of (acyclic) cluster algebras; we refer to Keller \cite{Keller_categorification} for an excellent introduction. (We note that we actually do not need cluster algebras directly, but only the categorification of quiver mutation, their combinatorial key structure. So we refer to the mentioned articles for the definition of cluster algebras and more information on the subject.) This relation yields an interplay between combinatorics and representation theory. For some problems concerning quiver mutations there are representation theoretic but no combinatorial solutions, for some obvious questions no answer has been found yet. For example it is not known how to decide whether two given quivers are mutation-equivalent (see Definition~\ref{dfn:mutequ}) or not. In fact this note is an example for this interplay concerning another question: Since acyclic quivers are of particular importance in representation theory, the question arises which quivers can be obtained from acyclic quivers via mutation. Using the said relation we give a certain necessary condition, namely all full subquivers of such a quiver must be of the same kind.

Shortly after finishing this note we found that this fact was already mentioned in \cite{BMRvia}. Unfortunately it was not yet included in the preprint version \href{http://de.arxiv.org/pdf/math.RT/0412077.pdf}{math.RT/0412077} that we had checked. In fact we use a result from this paper (but see also Remark~\ref{rem:hubery}), yet our proof is slightly more explicit.

The content is organised as follows. Section \ref{sec:result} contains the precise formulation of the main result Proposition \ref{prop:mac}, Section \ref{sec:tools} collects the necessary representation theoretic facts, finally in Section \ref{sec:proof} we give the proof of Proposition \ref{prop:mac} and some remarks.

\noindent\textbf{Acknowledgements.} I would like to thank Philipp Lampe, Daniel Rohleder and my advisor Dieter Happel for reading earlier versions of this note.

\section{Notation and main result}
\label{sec:result}
In this note a quiver is a (finite) directed graph, i.e.\ a finite set of vertices and a finite set of arrows starting and ending at vertices. Multiple arrows and loops are allowed. A full subquiver consists of a subset of the vertices and all arrows starting and ending in this subset. We always mean full subquivers.
\begin{dfn}
\label{dfn:quivmut}
Given a finite quiver $\G$ with vertex set $I$ without loops or two-cycles, for any vertex $k \in I$ we define a new quiver $\my{\G}$, the \textbf{mutation} of $\G$ in direction $k$, in three steps:
\begin{enumerate}
	\item For any two vertices $i$ and $j$ with $a$ arrows from $i$ to $k$ and $b$ arrows from $k$ to $j$ add $ab$ arrows from $i$ to $j$.
	\item Remove any two-cycles created in the first step. To be more precise: if there have been $c$ arrows from $j$ to $i$, delete $\min\{ab,c\}$ arrows in both directions.
	\item Change the orientation of all arrows incident to $k$ and replace $k$ with a new vertex $k^*$.
\end{enumerate}
\end{dfn}
It is easy to check that mutation is involutive in the sense that if $\G$ is a quiver as above with a vertex $k$, then $\m{k^*}{\my{\G}}\cong\G.$ (Note that $\my{\G}$ again contains neither loops nor two-cycles.)

\begin{dfn}
\label{dfn:mutequ}
Two quivers as above are called \textbf{mutation-equivalent} if one can be transformed into a quiver isomorphic to the other via a sequence of mutations. Since mutations are involutive this indeed defines an equivalence relation on the set of all finite quivers without loops and two-cycles. A quiver is called \textbf{mutation-acyclic} if it is mutation-equivalent to an acyclic quiver, else it is called \textbf{mutation-cyclic}.
\end{dfn}
The main result can now be stated as follows.
\begin{prop}
\label{prop:mac}
 Let $\Qu$ be a mutation-acyclic quiver and $\Qu'$ a full subquiver. Then $\Qu'$ is also mutation-acyclic.
\end{prop}

\section{Tools from representation theory}
\label{sec:tools}
We will prove the main result by means of representation theory. The necessary tools are introduced in this section. For unexplained terminology from representation theory and in particular tilting theory we refer to the literature (e.g.\ \cite{ARS,ASS1}).

Let $H$ be a finite-dimensional hereditary algebra with $n$ simples over an algebraically closed field $\K$. We consider the corresponding cluster category $\cc:=\cc_{H}$ as introduced in \cite{BMRRT}. We use the fact that it categorifies quiver mutation and allows us to apply representation theory. The connection is as follows. In $\cc$ we can consider the (basic) maximal rigid objects called cluster-tilting objects. We recall the relevant results from \cite{BMRRT}:
\begin{theorem}
\begin{enumerate}
 \item Each cluster-tilting object has precisely $n$ indecomposable summands and is induced by a tilting module over some hereditary algebra $H'$ derived equivalent to $H$ via the inclusion $\Mod H' \hookrightarrow \cc_{H'}\cong\cc$.
\item Let $T=\bigoplus_{i=1}^n T_i$ be a cluster-tilting object. Then for any $k$ there is precisely one indecomposable object $T_k^*\ncong T_k$ such that $T'=\bigoplus_{i\neq k} T_i \oplus T_k^*$ is again a cluster-tilting object.
\end{enumerate}
\label{thm:ct-objects}
\end{theorem}
With a cluster-tilting object $T$ in $\cc$ we associate the quiver $\cq{T}$ of the cluster-tilted algebra $\End_{\cc}(T)^{op}$. Theorem \ref{thm:ct-objects}(b) allows to define a mutation of cluster-tilting objects in the obvious way, namely we say that $T'$ is the mutation of $T$ in direction $T_k$. For us the following result is crucial.
\begin{theorem}[\cite{BMRvia}]
\label{thm:clusterquivmut}
 In the above notation the quiver $\cq{T'}$ is obtained from $\cq{T}$ by mutating at the vertex $T_k$.
\end{theorem}
The fact that every cluster-tilting object $T$ in $\cc$ is induced by a tilting module over a hereditary algebra $H$ yields another description of the quiver $\cq{T}$ as follows. For a tilting module $T$ over $H$ denote by $\tq{H}{T}$ the ordinary quiver of the tilted algebra $\End_{H}(T)^{op}$ and choose a minimal system of relations $R(T)$ such that $\End_{H}(T)^{op}\cong\K\tq{H}{T}/\left\langle R(T)\right\rangle $. Then we have the following result from \cite{ABS-cta}.
\begin{theorem}
\label{thm:relations}
 The quiver $\cq{T}$ is obtained from the quiver $\tq{H}{T}$ with relations $R(T)$ by replacing each relation with an arrow in the opposite direction.
\end{theorem}
Recall that the quiver of the tilted algebra $\End_{H}(T)^{op}$ is always acyclic (see e.g.\ \cite[VIII.3.4]{ASS1}), so the only cycles occurring in $\cq{T}$ come from relations in $\tq{H}{T}$. If $\End_{H}(T)^{op}$ is hereditary there are no relations and $\cq{T}=\tq{H}{T}$.

The mutation of cluster-tilting objects yields a so-called exchange graph. Its vertices are given by the cluster-tilting objects and edges correspond to mutations. The key idea of the proof is to interpret a sequence of quiver mutations as a walk in the exchange graph and vice versa. For this we need the following result from \cite{CK2}.
\begin{theorem}
\label{thm:connected}
 Let $T$ be a cluster-tilting object and $T_v$ an indecomposable summand. Then the set of cluster-tilting objects containing $T_v$ as a summand form a connected subgraph $lk(T_v)$ of the exchange graph.
\end{theorem}

\section{The proof and further remarks}
\label{sec:proof}
\begin{proof}[Proof of Proposition \ref{prop:mac}]
 It clearly suffices to prove the claim for subquivers $\Qu'$ with $n-1$ vertices when $\Qu$ has $n$ vertices. So let $\Qu=\m{\underline{i}}{\Qu_a}$ be obtained from the acyclic quiver $\Qu_a$ by a sequence of mutations $\mu_{\underline{i}}$. Consider the hereditary algebra $H':=\K\Qu_a$ and the corresponding cluster category $\cc:=\cc_{H'}$. Of course $H'$ itself is a tilting module and induces a cluster-tilting object in $\cc$ with $\cq{H'}=\tq{H'}{H'}=\Qu_a$. So by Theorem \ref{thm:clusterquivmut} we can apply $\mu_{\underline{i}}$ to $H'$ and obtain a cluster-tilting object $T$ with $\cq{T}\cong \Qu$. Now let $\Qu'$ be a subquiver of $\Qu$ with $n-1$ vertices and consider the summand $T_v$ of $T$ corresponding to the last vertex. If we now keep $T_v$ and only allow to mutate the other summands of $T$ this amounts to mutations of $\Qu'$ since the arrows incident to $T_v$ have no influence on the number of arrows between two other summands.

By Theorem \ref{thm:ct-objects}(a) we can regard $T$ as a tilting module over a hereditary algebra $H$. Now let $B$ be the Bongartz complement to $T_v$ in $\Mod H$ and set $C:=T_v \oplus B$. Again $C$ induces a cluster-tilting object in $\cc$ and as $lk(T_v)$ is connected by Theorem \ref{thm:connected} there is a sequence of mutations $\mu_{\underline{j}}$ that does not exchange $T_v$ such that $C=\m{\underline{j}}{T}$. So $\m{\underline{j}}{\Qu'}$ is the quiver obtained from $\cq{C}$ by deleting the vertex $T_v$. Denote this quiver by $\Qu_B$. It suffices to show that $\Qu_B$ is acyclic.

To do this we distinguish two cases: First suppose that $T_v$ is projective in $\Mod H$. Then clearly $C\cong{ _HH}$, so $\End_{H}(C)^{op}\cong H$ is hereditary and $\cq{C}=\tq{H}{C}$ is acyclic. Thus $\Qu_B$ is acyclic as a subquiver of an acyclic quiver. Now assume that $T_v$ is not projective. It is well known (see e.g.\ \cite[III.6.4]{happel-buch}) that in this case $\Ho[H]{T_v,B}=0$ and $\End_H(B)$ is hereditary. We claim that this implies that we have no relations among the summands of $B$. For a relation between two summands $B_i$ and $B_k$ of $B$ would either arise from two maps $B_i\to T_v \to B_k$ or from two maps $B_i\to B_j \to B_k$ for a third summand $B_j$ of $B$. But the former contradicts $\Ho[H]{T_v,B}=0$ and the latter is impossible since $\End_H(B)$ is hereditary. So by passing from $\tq{H}{C}$ to $\cq{C}$ no arrows among the summands of $B$ are added and the subquiver $\Qu_B$ remains acyclic.
\end{proof}

\begin{rem}
\label{rem:hubery}
 Let us mention that a first proof did not use the cluster category but worked with Hubery's support-tilting modules or tilting pairs instead of cluster-tilting objects, see \cite{hubery-cc,hubery}. In \cite{hubery} Hubery associates with any tilting pair a certain matrix (corresponding to the quivers $\cq{T}$ above) and shows that mutation of tilting pairs is compatible with matrix mutation (corresponding to Theorem \ref{thm:clusterquivmut}). Using results of \cite{hubery-cc} analogous to Theorem \ref{thm:connected} the crucial point is to show that a submatrix associated with the Bongartz complement of an indecomposable summand encodes an acyclic quiver.
\end{rem}

\begin{rem}
 A reformulation of Proposition \ref{prop:mac} is the following. If a given quiver $\Qu$ contains a mutation-cyclic subquiver, then $\Qu$ itself is mutation-cyclic. As it is easy to check whether a three-point-quiver is mutation-cyclic or not (see \cite{3ptmut}), this might be useful for showing that a given quiver is mutation-cyclic.
\end{rem}

\begin{rem}
 By the main result the condition that all proper subquivers are mutation-acyclic is necessary for a quiver to be mutation-acyclic. But it is not sufficient by the following easy counterexample. Consider the cyclic quiver with three vertices and doubled arrows. This is stable under mutation and thus mutation-cyclic, but all its proper subquivers are even acyclic.
\end{rem}

\bibliographystyle{alpha}
\bibliography{/HOME1/users/personal/warkm/Mathe/Darstellungstheorie/Latex-Projekte/bibliography/literatur}

\end{document}